\documentclass[11pt]{amsart}
\usepackage[margin=1in]{geometry}

\usepackage{physics, enumitem}
\usepackage[colorlinks=true, citecolor=blue]{hyperref}

\newcommand{\subeq}{\subseteq}

\newcommand{\ol}{\overline}

\newcommand{\bigp}[1]{\left( #1 \right)} % (x)
\newcommand{\bigb}[1]{\left[ #1 \right]} % [x]
\newcommand{\bigc}[1]{\left\{ #1 \right\}} % {x}
\newcommand{\bE}{\mathbb E}
\newcommand{\bN}{\mathbb N}
\newcommand{\bP}{\mathbb P}
\newcommand{\bR}{\mathbb R}

\newtheorem{thm}{Theorem}[section]
\newtheorem{theorem}[thm]{Theorem}
\newtheorem{corollary}[thm]{Corollary}
\newtheorem{lemma}[thm]{Lemma}

\newtheorem{definition}[thm]{Definition}

\newtheorem{conjecture}[thm]{Conjecture}
\newtheorem*{claim}{Claim}

\newtheorem{remark}[thm]{Remark}

\DeclareMathOperator{\ept}{ept}
\DeclareMathOperator{\thpzf}{th_{pzf}}
\DeclareMathOperator{\ptpzf}{pt_{pzf}}

\title{Bounds on expected propagation time of probabilistic zero forcing}
\author{Shyam Narayanan}
\address[Shyam Narayanan]{Department of Mathematics, Massachusetts Institute of Technology}
\email{shyamsn@mit.edu}

\author{Alec Sun}
\address[Alec Sun]{Department of Mathematics, Harvard University}
\email{sundogx@gmail.com}

\begin{document}

\begin{abstract}
    Probabilistic zero forcing is a coloring game played on a graph where the goal is to color every vertex blue starting with an initial blue vertex set. As long as the graph $G$ is connected, if at least 1 vertex is blue then eventually all of the vertices will be colored blue. The most studied parameter in probabilistic zero forcing is the expected propagation time $\text{ept}(G)$. We significantly improve on upper bounds for $\text{ept}(G)$ by Geneson and Hogben and by Chan et al.\! in terms of a graph's order and radius. We prove the bound $\text{ept}(G) = O\left(r\log\frac{n}{r}\right).$ We also show using Doob's Optional Stopping Theorem that $\text{ept}(G) \le \frac{n}{2} + O(\log n).$ Finally, we derive an explicit lower bound $\text{ept}(G)\ge \log_2 \log_2 (2n).$
\end{abstract}

\maketitle

\section{Introduction}
Zero forcing is an iterative coloring process on a graph. The concept was introduced in order to attack the maximum nullity problem of combinatorial matrix theory \cite{max-null2, max-null, max-null3, max-null4}, as well as independently to study quantum system control \cite{quantum}. Let $G$ denote a graph with vertex set $V$ and edge set $E$. As is standard, define the \emph{order} of $G$ to be $\abs{V}$ and the \emph{radius} of $G$ to be the minimum graph eccentricity of a vertex $v\in V$. Denote the order of $G$ by $n$ and the radius of $G$ by $r$. In this paper, we will assume that $G$ is connected and that $n\ge 3$. Zero forcing is described by the following algorithm. Let each vertex of a graph $G$ be either blue or white. Denote by $S$ the initial set of blue vertices of $G.$ The zero forcing color change rule changes the color of a vertex $v$ from white to blue if $v$ is the only white neighbor of a blue vertex $u.$ In this case, we say that $u$ forces $v$ and we write $u\rightarrow v.$ The initial blue set $S$ is said to be \emph{zero forcing} if, after finitely many steps of the color change rule, all vertices of $G$ are forced to blue. The \emph{zero forcing number} of $G,$ denoted as $Z(G),$ is defined as the minimum cardinality of a zero forcing set of $G.$

Viewing zero forcing as a dynamical process on a graph, Chilakamarri et al.\! \cite{Chilakamarri}, Fallat et al.\! \cite{prop}, and Hogben et al.\! \cite{prop2} have studied the number of steps it takes for an initial vertex set to force all other vertices to blue. This is called the \emph{propagation time} of a zero forcing set. In the context of quantum systems, the propagation time is also called the \emph{graph infection number} \cite{inf}. Zero forcing was later found to have connections with power domination \cite{power-dom} and graph searching \cite{graph-search}.

\subsection{Probabilistic Zero Forcing} \label{intro}

Probabilistic zero forcing is a modified zero forcing process first proposed by Kang and Yi \cite{KY}. Given a current set $B$ of blue vertices, each vertex $u\in B$ attempts to force each of its white neighbors $v \in \ol{B}$ blue independently with probability $$\Pr[u\to v] = \frac{\abs{N[u]\cap B}}{\deg u},$$ where $N[u]$ denotes the closed neighborhood of $u$, that is, includes $u$. This is known as the \emph{probabilistic color change rule} \cite{KY}. Repeated applications of this color change rule is known as \emph{probabilistic zero forcing}. We remark that while classical zero forcing is a deterministic process, probabilistic zero forcing is randomized. Note that probabilistic zero forcing reduces to classical zero forcing when a blue vertex $v$ has exactly 1 white neighbor $w.$ In this case $w$ will be deterministically forced in zero forcing and forced with probability $1$ in probabilistic zero forcing.

Probabilistic zero forcing is a discrete dynamical system that may be able to model certain problems better than classical zero forcing. For instance, the authors of \cite{GH} note that zero forcing is sometimes used to model rumor spreading in social networks, but given sporadic human nature a probabilistic model is more realistic. The spread of infection among a population, or the spread of a computer virus in a network, is better modeled probabilistically as well. As noted in \cite{snipedround2}, probabilistic zero-forcing is very similar to the well-studied \emph{push} and \emph{pull} models for rumor spreading from theoretical computer science \cite{5e, 13e}. For the \emph{push} model, one starts with a set of blue vertices, and at each time step, each blue vertex chooses 1 neighbor independently and uniformly at random and forces that vertex blue, if that vertex is white. For the \emph{pull} model, at each time step each white vertex chooses a neighbor independently and uniformly at random, and the white vertex turns blue if the chosen neighbor is blue. The two models can also be combined to create a \emph{push and pull} model in which at each time step, blue vertices choose a random neighbor to force and white vertices choose a random neighbor to try to become blue.

% As another example, consider a number of neurons in a graph, and suppose that every activated neuron fires a signal to each of its adjacent vertices. If the neuron is activated, it has some probability of activating its neighbor neurons, and if the activated neuron has many activated neighbors, it further increases the probability of activating its neighbor neurons because of ``positive feedback.'' Probabilistic zero forcing is a potential candidate for modelling a neuron system.

Just as propagation time is studied in classical zero forcing, a natural parameter of interest is the \emph{expected propagation time} of a vertex set in probabilistic zero forcing, which we will now define. The propagation time of a nonempty set $S$ of vertices of $G,$ denoted as $\ptpzf(G, S),$ is a random variable that represents the time at which the last white vertex turns blue when applying a probabilistic zero forcing process starting with the set $S$ blue. For a set $S\subeq V$ of vertices, the expected propagation time of $S$ for $G$ is the expected value of the propagation time of $S$, namely $\ept(G, S) = \bE[\ptpzf(G, S)].$ We are especially interested in the case where $\abs{S} = 1,$ hence we define the expected propagation time $\ept(G)$ for $G$ as the minimum expected propagation time of a single vertex, namely $\ept(G) = \min_{v \in V} \bE[\ptpzf(G, \{v\})].$
Another parameter in the study of zero forcing is the \emph{throttling number}. Throttling was initially defined in \cite{throttle} in order to study the balance between resources used to accomplish a task and time needed to accomplish the task. In \cite{GH}, the throttling number of a set $S$ of vertices of $G$ is defined to be $\thpzf(G, S) = \abs{S} + \ept(G, S),$ and the throttling number of a graph $G$ is defined to be $\thpzf(G) = \min_{S \subseteq V} \bigc{\thpzf(G, S)}.$

\subsection{Previous Results}

    For functions $f$ and $g$, we write $f=O(g)$ if there exists some absolute constant $c$ such that $f \le cg$, $f = \Omega(g)$ if $g = O(f)$, and $f = \Theta(g)$ if $f = O(g)$ and $f = \Omega(g)$. We also write $f = o(g)$ if $\lim_{n \to \infty} \frac{f(n)}{g(n)} = 0$ and $f = \omega(g)$ if $g = o(f)$.

In \cite{GH}, the authors studied the expected propagation time $\ept(G)$ of various families of graphs and showed the following results.

\begin{theorem}[{\cite[Proposition 2.2, Theorem 2.7, Corollary 2.9, and Theorem 2.11]{GH}}] \label{GHThm}
Let $n \ge 3$.
\begin{enumerate}[label=(\alph*)]
    \item For $P_n$, the path graph on $n$ vertices, $$\ept(P_n) = \begin{cases} \frac{n}{2} + \frac{2}{3} & n\equiv 0\bmod 2 \\ \frac{n}{2} + \frac{1}{2} & n\equiv 1\bmod 2.\end{cases}$$
    \item For $K_{1, n}$ defined as the star graph with $n$ leaves, $\ept(K_{1, n}) = \Theta(\log n)$.
    \item For $K_n$, the complete graph on $n$ vertices, $\Omega(\log \log n) \le \ept(K_n) \le O(\log n).$ 
    \item We have $\ept(G) = O(r \log^2 n).$
\end{enumerate}
\end{theorem}

In \cite{sniped}, the authors used Markov chains to explicitly compute the expected propagation time of several graph families, including complete graphs, complete bipartite graphs, ``sun'' graphs, and ``comb'' graphs. In addition, they prove the following two results.
\begin{theorem}[{\cite[Theorem 3.1 and Theorem 3.4]{sniped}}] \label{snipedThm}
Let $n \ge 3.$
\begin{enumerate}[label=(\alph*)]
    \item For $K_n$, the complete graph on $n$ vertices, $\ept(K_n) = \Theta(\log \log n).$
    \item We have $\ept(G) \le \frac{e}{e-1}\cdot n.$
\end{enumerate}
\end{theorem}

Finally, in \cite{snipedround2}, the authors studied probabilistic zero forcing on Erd\H{o}s--R\'{e}nyi random graphs and proved the following result.

\begin{theorem}[{\cite[Theorem 1.3]{snipedround2}}] \label{snipedround2Thm}
    Let $G = G(n, p)$ denote the Erd\H{o}s--R\'{e}nyi random graph where each edge is independently in the graph with probability $p = p(n)$ with $p(n) = \omega\bigp{\frac{\log n}{n}}$. Then, for all vertices $v \in V,$ the following holds with probability approaching $1$ as $n\to \infty$:
\begin{align*}
    \ptpzf(G, v) &\le (1 + o(1)) \cdot \left(\log_2 \log_2 n + \log_3 (1/p)\right) \\
    \ptpzf(G, v) &\ge (1 - o(1)) \cdot \max\left(\log_2 \log_2 n, \log_4 (1/p)\right).
\end{align*}
\end{theorem}

\subsection{Paper Outline}

In Section \ref{prelim} we give some more definitions from \cite{GH} related to propagation time, and we also introduce the statistical tools that we will be using in the subsequent proofs.

In Section \ref{radius} we prove that $\ept(G) = O\bigp{r\log\frac{n}{r}}.$ This improves on the bound $\ept(G) = O(r\log^2 n)$ from Theorem \ref{GHThm} (d) due to Geneson and Hogben \cite{GH}. We also prove tightness of this bound up to a multiplicative constant for a family of graphs.

In Section \ref{size}, we prove the upper bound $\ept(G) \le \frac{n}{2} + O(\log n)$ for the expected propagation time.  This improves on the bound $\ept(G) \le \frac{e}{e-1}\cdot n$ of Theorem \ref{snipedThm} (b) due to Chan et al.\! \cite{sniped}, and is asymptotically tight up to a multiplicative factor of $1+o(1)$ for the path graph by Theorem \ref{GHThm} (a). We then prove an explicit lower bound $\ept(G) \ge \log_2 \log_2 (2n)$ for the expected propagation time of $G$. We also derive as a corollary that the throttling number has the lower bound $\thpzf(G)\ge \log_2 \log_2 (2n).$ This lower bound is tight up to a multiplicative constant for the complete graph $K_n$ by Theorem \ref{snipedround2Thm} (a).

In Section \ref{Conclusion}, we provide some further open problems and conjectures for expected propagation times of different graph families.

\subsection*{Acknowledgements}
This research was funded by NSF/DMS grant 1659047 and NSA grant H98230-18-1-0010 as part of the 2019 Duluth Research Experience for Undergraduates (REU) program. The authors thank Joseph Gallian for suggesting the problem. The authors are grateful to Joseph Gallian, Trajan Hammonds, and several anonymous reviewers for many suggestions that improved the presentation of this article. Finally, the authors especially thank an anonymous reviewer who went above and beyond the call of duty to provide detailed and thorough comments.

\section{Preliminaries}\label{prelim}

\subsection{Definitions}

For any subset $S \subeq V,$ denote by $G[S]$ the subgraph of $G$ restricted to the vertex set $S$.

\begin{definition}
    For an undirected graph $G$ of blue and white vertices, we define $\deg v$ as the total degree of any vertex $v \in V,$ $\deg_w v$ as the number of white neighbors of $v$, and $\deg_b v$ as the number of blue neighbors of $v$. For any subset $S \subeq G,$ define $\deg_S v$ as the number of neighbors of $v$ in $S$.
\end{definition}

Recall from Section \ref{intro} that $\ept(G, S)$ is the expected propagation time starting with $S$ blue and that $\ept(G) = \min_{v \in V} \ept(G, \{v\})$. In this paper, we are also concerned with the probability that some subset of vertices is blue at a particular time step $t$. Therefore, we make the following definitions, letting $S$ denote the initial blue set.

\begin{definition}
    For a subset $T \subeq V,$ define $P^{(t)}(G, S, T)$ as the probability that after $t$ steps of probabilistic zero forcing, all vertices in $T$ are blue. Define $P^{(t)}(G,S) = P^{(t)}(G,S,V)$ to be the probability that after $t$ steps, all vertices of $G$ are blue. Define $\ept(G, S, T)$ as the expected number of steps needed until all vertices in $T$ are blue, noting that $\ept(G, S) = \ept(G, S, V)$.
\end{definition}

% \begin{remark}
%     Note that $P^{(t)}(G, S, V) = P^{(t)}(G, S)$ and that $\ept(G, S, V) = \ept(G, S).$
% \end{remark}

\subsection{Tools from Probability Theory}

We will use some well-known concentration inequalities.

\begin{theorem}[Markov's Inequality]
    Given a non-negative random variable $X$ with expectation $\bE[X],$ we have for all $\lambda > 0$ that $$\bP(X \ge \lambda) \le \frac{\bE[X]}{\lambda}.$$
\end{theorem}

\begin{theorem}[Chebyshev's Inequality]
    Given a random variable $X$ with expectation $\bE[X]$ and variance $\mathrm{Var}(X),$ we have for all $\lambda > 0$ that $$\bP\bigp{\abs{X - \bE[X]} \ge \lambda} \le \frac{\mathrm{Var}(X)}{\lambda^2}.$$
\end{theorem}

\begin{theorem}[Chernoff Bound]
    Let $X_1, \dots, X_n$ be independent random variables taking values in $\{0, 1\}$. Let $X = \displaystyle \sum _{i=1}^n X_i$ and denote $\mu = \bE[X]$. Then for any $\delta > 0$, we have
\[\bP(X > (1+\delta) \mu) < \left(\frac{e^{\delta}}{(1+\delta)^{1+\delta}}\right)^{\mu}.\]
    If $0 < \delta < 1,$ then
\[\bP(X < (1-\delta) \mu) < \left(\frac{e^{-\delta}}{(1-\delta)^{1-\delta}}\right)^{\mu}.\]
\end{theorem}

% We will also need the following concentration inequality, due to McDiarmid \cite{McDiarmid}.

% \begin{theorem}[McDiarmid's Inequality]
%     Suppose that $X_1, \dots, X_n$ are independent real-valued random variables and $f(X_1, \dots, X_n)$ is a function from $\BR^n$ to $\BR$, such that for all $i$, there exists some constant $c_i$ such that we always have
% \[\abs{f(X_1, \dots, X_i, \dots, X_n) - f(X_1, \dots, X_i', \dots, X_n)} \le c_i.\]
%     Then, we have that
% \[\bP_{X_1, \dots, X_n} \left(\abs{f(X_1, \dots, X_n) - \bE f(X_1, \dots, X_n)} \ge t\right)\le \exp\left(-\frac{2 t^2}{\sum_{i = 1}^{n} c_i^2}\right).\]
% \end{theorem}

We will also need some results from martingale theory. First, we state some definitions.

\begin{definition}
    A sequence of random variables $M_0,M_1,\ldots$ with finite absolute means is called a \emph{martingale} with respect to another sequence of random variables $X_0,X_1,\ldots$ if for all $n,$ $M_n$ is a function of $X_0,X_1\ldots,X_n$ and $\bE[M_{n+1}\mid X_0,X_1,\ldots,X_n] = M_n.$
    The sequence $M_0,M_1,\ldots$ is called a \emph{submartingale} if all conditions are the same, except $\bE[M_{n+1}\mid X_0,X_1,\ldots,X_n] \ge M_n.$
    The sequence $M_0,M_1,\ldots$ is called a \emph{supermartingale} if all conditions are the same, except $\bE[M_{n+1}\mid X_0,X_1,\ldots,X_n] \le M_n.$
\end{definition}

\begin{definition}[Stopping Time]
    A random variable $T$ taking values in $\{0,1,2,\ldots\}$ is called a \emph{stopping time} with respect to $X_0,X_1,\ldots$ if for each $n,$ the indicator of the event $T\le n$ is a measurable function of $X_0,X_1,\ldots,X_n.$ That is, $\{T\le n\}\in \sigma(X_0,\ldots,X_n)$ for all $n.$
\end{definition}

\begin{remark}
A random variable $T$ being a stopping time means that it is known at time $n$ whether $T\le n.$ In other words, one cannot look at the ``future'' variables $X_{n+1}, X_{n+2}, \dots$ to decide whether or not to stop.
\end{remark}

We use a formulation of Doob's Optional Stopping Theorem found in \cite{textbook}.

\begin{theorem}[Doob's Optional Stopping Theorem]
    Suppose that $M_n$ is a martingale with respect to $X_n$ and that $T$ is a stopping time with respect to $X_n.$ Suppose that there exists a constant $c>0$ such that $\abs{M_n-M_{n-1}}\le c$ for all $n$ and further assume that $\bE[T]<\infty.$ Then $\bE[M_T] = \bE[M_0].$ If $M_n$ is a submartingale and all else is equal, then $\bE[M_T] \ge \bE[M_0]$, and if $M_n$ is a supermartingale and all else is equal, then $\bE[M_T] \le \bE[M_0].$
\end{theorem}

\subsection{Coupling Results}

We will need some ``coupling results'' about probabilistic zero forcing, where we show that probabilistic zero forcing processes terminate more quickly than certain modified probabilistic zero forcing processes. We call these results ``coupling results'' because their proofs involve a technique called coupling. The idea behind coupling is that if one wants to prove that $\bE[X_1] \le \bE[X_2]$ for some random variables $X_1$ and $X_2,$ one defines random variables $Y_1$ and $Y_2$ such that $X_1$ and $Y_1$ have the same distribution and $X_2$ and $Y_2$ have the same distribution, but $Y_1 \le Y_2$ always. The variables $Y_1$ and $Y_2$ are defined by properly correlating the randomness used in creating $X_1$ and in creating $X_2.$ 

In \cite{GH}, the authors proved the following result:

\begin{lemma}[{\cite[Proposition 4.1]{GH}}] \label{GenesonCoupling}
    Suppose that $S \subeq T$. Then, $P^{(\ell)}(G, S) \le P^{(\ell)}(G, T).$ As an immediate corollary, we have $\ept(G, S) \ge \ept(G, T).$
\end{lemma}

Our paper requires a stronger result than Lemma \ref{GenesonCoupling}.

\begin{lemma} \label{OurCoupling}
    Suppose that initially, some set $S \subeq V$ is blue. Say that we follow some modified probabilistic process where at the $t^{\text{th}}$ step, $\bP_t[u \to v]$, the probability that $u$ converts $v$ to become blue at step $t$, is some function of $G, u, v,$ and $B_{t-1},$ the set of blue vertices after the $(t-1)^{\text{th}}$ step. In addition, suppose that $$\bP_t[u \to v] \le \frac{\abs{N[u] \cap B_{t-1}}}{\deg u}$$ for all blue vertices $u$ and white neighbors $v$ of $u$, and that conditioned on $u, v,$ and $B_{t-1}$, the set of events $u \to v$ is independent. Then, for any $T \subeq V$ and any $\ell \ge 1$, the probability that all vertices in $T$ are blue after time step $\ell$ is at most $P^{(\ell)}(G, S, T),$ the probability that all vertices in $T$ would be blue if we followed the normal probabilistic zero forcing process. Consequently, the expected amount of time until all vertices in $T$ are blue is at least $\ept(G, S, T).$
\end{lemma}

\begin{proof}
    Let $Q^{(\ell)}(G, S, T)$ be the probability that all vertices of $T$ are blue for the modified probabilistic process. The idea is to assign some common randomness to both the normal probabilistic zero forcing process and the modified process, and then show that if we condition on the randomness, the set of blue vertices in the modified process is a subset of the blue vertices in the normal probabilistic zero forcing process after each step. Let $e \in V\times V$ be a directed edge of $G$, where we let $(u, v)$ and $(v, u)$ be directed edges if the undirected edge $(u, v)$ is in $G$. For each directed edge $e$ and each time step $t$, we will create a random variable $X_{e, t} \sim \text{Unif} [0, 1]$. Each $X_{e, t}$ will be independent and uniformly distributed between $0$ and $1$.
    
    Consider the following modified process. At each step $t$ and for each edge $e = (u, v),$ if $u$ is blue and $v$ is white after step $t-1$, we will convert $v$ to blue if $X_{e, t} \le \bP_t[u \to v]$. In other words, $v$ will be blue after time step $t$ if $v$ was blue after time step $t-1,$ or there is some directed edge $(u, v)$ such that $u$ was blue at time step $t-1$ and $X_{(u, v), t} \le \bP_t[u \to v].$ Likewise, in the normal probabilistic zero forcing process, if $u$ is blue and $v$ is white, we will convert $v$ to blue if $$X_{(u, v), t} \le \frac{\abs{N[u] \cap S_{t-1}}}{\deg u},$$ where $S_{t-1}$ is the set of blue vertices in the normal process after the $(t-1)^{\text{th}}$ step. Since $$\bP(X_{e, t} \le \bP_t[u \to v]) = \bP_t[u \to v]$$ and $$\bP\left(X_{(u, v), t} \le \frac{\abs{N[u] \cap S_{t-1}}}{\deg u}\right) = \frac{\abs{N[u] \cap S_{t-1}}}{\deg u},$$ the processes we are following indeed are correct. Therefore, it suffices to show that $B_{t} \subeq S_{t}$ for all $t$. We prove this by induction. For $t = 0,$ $B_t = S_t = S,$ so it is clear. If true at some time step $t-1$, then we have to show that if $v$ is blue after time step $t$ in the modified process, then it is also blue in the original process. For any $v$ that is white after step $t$ in the normal process, $v \not\in S_{t},$ so $v \not\in S_{t-1},$ which means $v \not\in B_{t-1}$ by our induction hypothesis. Then, $$X_{(u, v), t} > \frac{\abs{N[u] \cap S_{t-1}}}{\deg u} \ge \frac{\abs{N[u] \cap B_{t-1}}}{\deg u}$$ for all $u \in B_{t-1}$ connected to $v$, or else $v$ would become blue at step $t$ of the normal process. Therefore, $X_{(u, v), t} > \bP_t[u \to v]$ for all $u \in B_{t-1}$ connected to $v$, so $v$ remains white after step $t$. This completes the induction. Since $Q^{(\ell)}(G, S, T) = \bP(T \subeq B_{\ell})$ and $P^{(\ell)}(G, S, T) = \bP(T \subeq S_{\ell}),$ we have that $Q^{(\ell)}(G, S, T) \le P^{(\ell)}(G, S, T)$ since $B_{\ell} \subeq S_{\ell}.$
    
    If we let $\ept_Q(G, S, T)$ denote the expected time until all vertices are blue in the modified process, then by the tail sum formula for expectation, we have $$\ept(G, S, T) = \sum_{\ell = 0}^{\infty} \left(1 - P^{(\ell)}(G, S, T)\right)$$
    and
$$\ept_Q(G, S, T) = \sum_{\ell = 0}^{\infty} \left(1 - Q^{(\ell)}(G, S, T)\right).$$
    We conclude that $\ept_Q(G, S, T) \ge \ept(G, S, T)$.
\end{proof}

We remark that the statement and proof of Lemma \ref{OurCoupling} are similar to those of \cite[Lemma 2.2]{snipedround2}.

\section{Radius bound for general graphs}\label{radius}

Our goal in this section is to prove the following theorem.

\begin{theorem}\label{radius-bound}
    We have $\ept(G) = O\left(r \log \frac{n}{r}\right)$.
\end{theorem}

We begin by proving some lemmas about probabilistic zero forcing on a star graph.

\begin{lemma} \label{LowIncrease}
    Let $H$ be a star graph with $n\ge 2$ leaves, with its center and exactly $k$ leaves colored blue and all other vertices colored white. Then, if $k \le \frac{n}{3},$ the number of leaves that will turn blue in the next step will be at least $\frac{k+1}{6}$ with probability at least $\frac{1}{5}.$
\end{lemma}

\begin{proof}
    Note that the subsequent step, each white leaf will become blue with probability $\frac{k+1}{n}.$ Thus, the number of leaves that will turn blue in the subsequent step, which we denote by the random variable $X$, has distribution $\text{Bin}\big(n-k, \frac{k+1}{n}\big).$ Note that since $k \le \frac{n}{3},$ then $$\bE[X]\ge \frac{2n}{3} \cdot \frac{k+1}{n} = \frac{2(k+1)}{3}.$$ We have
\[\mathrm{Var}(X) = (n-k) \cdot \frac{k+1}{n} \cdot \left(1 - \frac{k+1}{n}\right) \le (n-k) \cdot \frac{k+1}{n} = \bE(X).\] By Chebyshev's Inequality, we have
\[\bP\left(X < \frac{k+1}{6}\right) \le \frac{\mathrm{Var}(X)}{\bigp{\bE[X] - \frac{k+1}{6}}^2} \le \frac{\mathrm{Var}(X)}{\frac{9}{16} \bE[X]^2} \le \frac{16}{9\bE[X]} \le \frac{8}{3(k+1)} \le \frac{8}{21}\]
if $k\ge 6$. Therefore, if $k\ge 6$ the number of blue vertices increases by at least $\frac{k+1}{6}$ with probability at least $\frac{13}{21} \ge \frac{1}{5}.$ For the remaining cases $0 \le k \le 5,$ we use the following argument. We have
\[\bP(X = 0) = \bP\left(\text{\text{Bin}}\left(n-k, \frac{k+1}{n}\right) = 0\right) = \left(1 - \frac{k+1}{n}\right)^{n-k} \le e^{-(k+1) \cdot (n-k)/n}.\]
    Note that since $k+1 \ge 1$ and $\frac{n-k}{n} \ge \frac{2}{3},$ we have $e^{-(k+1) \cdot (n-k)/n} \le e^{-2/3}.$ Therefore, the number of blue vertices increases by at least $1 \ge \frac{k+1}{6}$, because $k\le 5$, with probability at least $1 - e^{-2/3} \ge \frac{1}{5}.$
\end{proof}

\begin{lemma} \label{HighIncrease}
    Again, let $H$ be a star graph with $n$ leaves, with its center and exactly $k$ leaves colored blue and all other vertices colored white. Then, if $k \ge \frac{n}{3},$ the number of leaves that will turn blue in the next step will be at least $\frac{n-k}{6}$ with probability at least $\frac{1}{5}.$
\end{lemma}

\begin{proof}
    Since $k \ge \frac{n}{3},$ each white leaf will turn blue with probability $\frac{k+1}{n} \ge \frac13.$ Therefore, the expected number of white leaves that remain white is at most $\frac{2}{3} (n-k),$ so by Markov's Inequality, the probability of there remaining at least $\frac{5}{6} (n-k)$ leaves that are white is at most $\frac{2(n-k)/3}{5(n-k)/6} = \frac{4}{5}.$ Thus, at least $\frac{n-k}{6}$ of the leaves will become blue with probability at least $\frac{1}{5}.$
\end{proof}

\begin{lemma} \label{StarExponentialTailBound}
    Let $H$ be a star graph with $n\ge 2$ leaves, with its center colored blue and all other vertices colored white. There exist explicit constants $C > 0$ and $0<\alpha<1$, independent of $n$, such that the blue vertex will propagate to all the leaves in $t$ steps with probability at least $1-\alpha^t,$ whenever $t > C \log n$.
\end{lemma}

\begin{proof}
    We partition the interval $[0, n]$ into subintervals as follows. Let $I_1 = [0, 1).$ If $I_j = [a_j, b_j)$ for $b_j < \frac{n}{3},$ we set $I_{j+1} = \Big[b_j, b_j + \frac{b_j+1}{6}\Big) \cap \Big[0, \frac{n}{3}\Big).$ If $I_J = [a_J, b_J)$ for $b_J = \frac{n}{3},$ then we set $I_{J+r} = \Big[n - \frac{2n}{3 \cdot 6^{r-1}}, n - \frac{2n}{3 \cdot 6^r} \Big)$ whenever $\frac{2n}{3 \cdot 6^{r-1}} \ge 1.$ For the least value $R$ such that $\frac{2n}{3 \cdot 6^{R-1}} < 1,$ we set $I_{J+R} = \Big[n - \frac{2n}{3 \cdot 6^{R-1}}, n\Big]$ to be the final interval. 
    
    Note that $I_1, \dots, I_J$ partition $\left[0, \frac{n}{3}\right)$ and $I_{J+1}, \dots, I_{J+R}$ partition $\left[\frac{n}{3}, n\right]$ so we have a complete partition. Moreover, it is straightforward to verify that $J \le C_1 \log n$ and $R \le C_2 \log n$ for some constants $C_1, C_2$, so $J+R \le (C_1+C_2) \log n.$ Also, note that by Lemma \ref{LowIncrease}, if the number of blue leaves is $k \in I_r$ for $r \le J,$ with probability at least $1/5$ the number of blue leaves will be in some $I_s$ for $s > r.$ Moreover, by Lemma \ref{HighIncrease}, the same is true for $J+1 \le r \le J+R-1.$
    
    Since $n$ is the only integer in $I_{J+R},$ the probability that all $n$ vertices are blue after $t$ steps is at least the probability that a random walk, that moves right with probability $\frac{1}{5}$ and is stationary otherwise, moves at least $J+R \le (C_1+C_2) \log n$ to the right after $t$ steps. Letting $C_3 = C_1+C_2,$ the probability that this random walk moves at least $J+R$ to the right after $t$ steps is at least $\bP(\text{\text{Bin}}(t, 1/5) \ge C_3 \log n).$ If $t\ge 10 C_3 \log n$ then $\bE[\text{\text{Bin}}(t, 1/5)] \ge 2 C_3 \log n,$ so by the Chernoff Bound,
\begin{align*}
\bP\left(\text{\text{Bin}}\left(t, \frac{1}{5}\right) < C_3 \log n\right) 
\le \bP\left(\text{\text{Bin}}\left(t, \frac{1}{5}\right) < \frac{t}{10}\right)
\le \left(\frac{e^{-\frac{1}{2}}}{(\frac{1}{2})^{\frac{1}{2}}}\right)^{t/5} = \left(\frac{2}{e}\right)^{t/10}.
\end{align*}

    Therefore, if we set $\alpha = (2/e)^{1/10}$ and $C = 10C_3,$ the probability that all leaves are blue after $t > C \log n$ steps is at least $1 - \alpha^{t}.$
\end{proof}

The next lemma contains most of the ingredients with which we can finish the main theorem of this section.

\begin{lemma} \label{RadiusBash}
    Fix a blue vertex $v$ of $G$. Let $w \neq v$ be some other vertex such that the shortest distance between $v$ and $w$ is $s$. Then there exist explicit constants $C, C' > 0$ and $0 < \beta < 1$ such that after $t + C' s + C s \log \frac{n}{s}$ steps, $w$ will be blue with probability at least $1 - \beta^t$.
\end{lemma}

\begin{proof}
    Choose some path $v = v_0, v_1, v_2, \dots, v_s = w$ such that $(v_{i-1}, v_i)$ is an edge for all $1 \le i \le s.$ For $0 \le i \le s,$ let $K_i$ denote the set of neighbors of $v_i$ that were not neighbors of $v_j$ for any $j < i$, and let $k_i = \abs{K_i}$. For $1 \le i \le s,$ let $X_i$ denote the amount of time it takes until all of $v_0, v_1, \dots, v_i$, as well as all neighbors of $v_0, \dots, v_{i-1}$ all become blue, and set $X_0 = 0$. Moreover, let $S_t$ be the set of blue vertices after time $t$.
    
    For $0 \le i \le s-1,$ consider the graph after $X_i$ steps, so $S_{X_i}$ is the set of blue vertices at this time. Consider a process where at time $X_i+t,$ if $w \in K_i$ but $w$ is white, we convert $w$ to blue with probability $\frac{1}{k_i}\cdot \bigp{1 + \abs{K_i \cap S_{X_i+t-1}}}$. Since these forcing probabilities match those in probabilistic zero forcing for a star graph with center $v_i$ and $k_i$ leaves which are the vertices in $K_i,$ by Lemma \ref{StarExponentialTailBound} and Lemma \ref{GenesonCoupling}, after $X_i + C \log k_i + t$ steps, all vertices in $K_i$ will be blue with probability at least $1 - \alpha^t,$ where $C, \alpha$ are the same as in Lemma \ref{StarExponentialTailBound}.
    
    However, in actual probabilistic zero forcing, at time step $X_i+t$ for $t \ge 1$, vertex $v_i$ will convert all of its white neighbors to blue with probability $$\frac{\abs{N[v_i] \cap S_{X_i+t-1}}}{\deg v_i} = \frac{1 + \abs{K_i \cap S_{X_i+t-1}} + (\deg v_i - k_i)}{k_i + (\deg v_i - k_i)} \ge \frac{1+\abs{K_i \cap S_{X_i+t-1}}}{k_i},$$ as all neighbors of $v_i$ not in $K_i$ were already blue by time $X_i$. Moreover, there may be additional vertices that are converting the vertices in $K_i$ to blue with some probability. Therefore, by Lemma \ref{OurCoupling}, after $X_i + C \log k_i + t$ steps, all neighbors of $v_i,$ including $v_{i+1}$ will be blue with probability at least $1 - \alpha^t,$ even if we condition on $X_i$ and $S_{X_i}.$ Therefore, for all $t \ge 0,$ $$\bP\left(X_{i+1}-X_i \le C \log k_i + t \mid X_i, S_{X_i}\right) \ge 1 - \alpha^t.$$
    Note that we can even condition on all previous $X_j$ for $j < i,$ since given $X_i$ and $S_{X_i},$ $X_{i+1}$ is independent of $X_0, \dots, X_{i-1}.$ To finish, we note that
\begin{align*}
\bE\left[\alpha^{-(X_{i+1}-X_i)/2} \mid X_0, \dots, X_i\right] &\le \alpha^{-C \log k_i/2} \cdot \left(1 + \sum\limits_{t = 1}^{\infty} \bP(X_{i+1}-X_i = C \log k_i + t) \alpha^{-t/2}\right) \\
&\le \alpha^{-C \log k_i/2} \cdot \left(1 + \sum\limits_{t = 1}^{\infty} \alpha^{t-1} \alpha^{-t/2}\right)
\\&= C_1 \alpha^{-C \log k_i/2}
\end{align*}
    for $$C_1 = 1 + \sum_{t = 1}^{\infty} \alpha^{t-1} \alpha^{-t/2} < \infty.$$ This means that
\begin{align*}
\bE\left[\alpha^{-X_s/2}\right] &\le \prod\limits_{i = 1}^{s} \bE\left[\alpha^{-(X_i-X_{i-1})/2} \mid X_0, \cdots, X_{i-1}\right] \\
&\le C_1^s \cdot \alpha^{-C (\sum \log k_i)/2} \\
&= \alpha^{-\left(\frac{s \log C_1}{\log(1/\alpha) } + \frac{C}{2} \bigp{\sum_i \log k_i}\right)}.
\end{align*}
    If we set $C' = 2\frac{\log C_1}{\log (1/\alpha)}$ and $\beta = \sqrt{\alpha}$, then by Markov's Inequality,
\begin{align*}
\bP\left(X_s \ge C' s + C \sum \log k_i + t\right) &= \bP\left(\alpha^{-X_s/2} \ge \alpha^{-(C' s + C \sum_i \log k_i + t)/2}\right) \\
&\le \frac{\bE\left[\alpha^{-X_s/2}\right]}{\alpha^{-(C' s + C \sum \log k_i)/2} \cdot \alpha^{-t/2}} \\&\le \alpha^{t/2} \\&= \beta^t.
\end{align*}
    No vertex can be in more than a single $K_i,$ so $\sum_i k_i \le n.$ Therefore, by Jensen's Inequality we have $\sum_i \log k_i \le s \log \frac{n}{s}.$ This concludes the proof of the lemma, as after time $X_s,$ $w$ is blue.
\end{proof}

We now finish the proof.

\begin{proof}[Proof of Theorem \ref{radius-bound}]
    Choose a starting vertex $v$ such that all vertices $w$ are of distance at most $r$ from $v$. For any fixed $w$ of distance $s$ from $v$, we have that after $Cs \log \frac{n}{s} + C's + t$ steps, $w$ will be blue with probability at least $1 - \beta^t$ by Lemma \ref{RadiusBash}. We note that $x \log \frac{n}{x}$ is increasing for $x \in \left[1, \frac{n}{e}\right)$ and decreasing on $\left(\frac{n}{e}, n\right],$ so if $r \le \frac{n}{e},$ then 
\[C\cdot s \log \frac{n}{s} + C's \le C\cdot r \log \frac{n}{r} + C' r \le (C+C')\cdot  r \log \frac{n}{r}.\]
If $\frac{n}{e} \le r \le \frac{n}{2},$ then $r \log \frac{n}{r} \ge \frac{n}{2} \log 2$, so
\[C\cdot s \log \frac{n}{s} + C's \le C \cdot \frac{n}{e}\cdot  \log\bigp{ \frac{n}{n/e}} + C'\cdot  \frac{n}{2} \le \frac{C+C'}{2}\cdot  n \le \frac{C+C'}{\log 2}\cdot  r \log \frac{n}{r}.\]
    Therefore, there is some constant $C_2 = \frac{C+C'}{\log 2}$ such that for any vertex $w$, $w$ will be blue with probability at least $1 - \beta^t$ after $C_2 \cdot r \log \frac{n}{r} + t$ steps. After $C_2 \cdot r \log \frac{n}{r} + \frac{\log n}{\log (1/\beta)} + t$ steps, each vertex $w$ will be blue with probability at least $1 - \frac{\beta^t}{n}$, so the entire graph will be blue with probability at least $1 - \beta^t.$ As $\beta < 1$ is a fixed constant, we have
\[\ept(G) \le C_2 \cdot r \log \frac{n}{r} + \frac{\log n}{\log (1/\beta)} + O(1) = O\left(r \log \frac{n}{r}\right),\]
    as desired.
\end{proof}

We end with a proof that the bound in Theorem \ref{radius-bound} is tight for the following family of graphs.

\begin{theorem}\label{rtight}
    There exists an absolute constant $c > 0$ such that for all positive integers $s, n$ with $2s+1\mid n$, one can find an example of a graph $G$ with radius $r = s+2$ such that $\ept(G) \ge c\cdot r\log\frac{n}{r}.$
\end{theorem}

\begin{proof}
    Construct the following graph $G$:
    \begin{enumerate}
        \item Generate $2s+1$ identical star graphs $G_{-s}, G_{-s+1}, \ldots, G_s$ each with $\frac{n}{2s+1}$ vertices.
        \item Arrange the centers $C_{-s},C_{-s+1},\ldots,C_{s}$ of the star graphs in a line and connect the vertices $C_j$ and $C_{j+1}$ with an edge for all $j = -s,-s+1,\ldots,s-1.$
    \end{enumerate}
    Note that a leaf of $G_0$ minimizes the eccentricity, hence $r = s+2$. We will lower bound the expected propagation time with an initial vertex $v$ colored blue. Without loss of generality suppose that $v$ is a vertex of $G_i$ with $i \le 0$.
    \begin{claim}
    For a star on $n$ vertices with a distinguished leaf $\ell,$ suppose the initial blue vertex set is either the center only or the center and a leaf that is not $\ell$. Then the probability that it takes $\Omega(\log n)$ time for $\ell$ to be turned blue is at least $\frac{1}{3}.$
    \end{claim}
    To prove the claim, we follow the same argument as \cite[Theorem 2.7]{GH}. We do the case where only the center is colored blue first. The proof of \cite[Theorem 2.7]{GH} tells us that given $b$ current blue vertices where $\sqrt{n} \le b \le \frac{n}{2}$, with probability at least $1-O\bigp{\frac{1}{\sqrt{n}}}$ the next step will have at most $5b$ blue vertices. Hence using this proof, with probability at least $1-o(1)$ it takes $\Omega(\log n)$ steps for the number of blue vertices to increase from below $\sqrt{n}$ to a value in the range $\bigb{\frac{n}{10}, \frac{n}{2}}.$ At this point, by symmetry of the star graph, with probability at least $\frac{1}{2}$ the leaf $\ell$ will not be colored blue. The case where the initial blue vertex set consists of the center and a leaf that is not $\ell$ is done using the same argument with $b=2$ initially rather than $b=1$.
    
    Returning to the proof of Theorem \ref{rtight}, consider the stars $G_i,G_{i+1},\ldots,G_s$ in sequence. Note that the vertex $C_k$ must be colored blue before $C_{k+1}$ is colored blue. For each graph $G_k$ for $k<s$, set $\ell$ in the claim to be $C_{k+1},$ which is a vertex of $G_k.$ If $v$ is a leaf of $G_i$, then after a single step both $v$ and $C_i$ will be colored blue. Otherwise, we are in the initial condition of the claim. For each step of the process that colors $C_i, C_{i+1}, \ldots, C_r$ blue in sequence, the claim tells us that with probability at least $\frac{1}{3}$ a total of $\Omega\bigp{\log\frac{n}{2s+1}}$ steps are needed for $C_{j+1}$ to become blue given that $C_j$ has just been colored blue. By a standard Chernoff Bound argument, with exponentially high probability we need $s\cdot \Omega\bigp{\log \frac{n}{2s+1}} = \Omega\bigp{r\log \frac{n}{r}}$ steps to color all the vertices blue, implying that $\ept(G) \ge c\cdot r\log\frac{n}{r}$ for some absolute constant $c$.
\end{proof}

\section{Size bounds for general graphs}\label{size}

In this section, we prove sharp upper and lower bounds for the maximum and minimum possible values of $\ept(G)$.

\subsection{An upper bound on expected propagation time}\label{upper}

In this section we prove that $\ept(G) \le \frac{1}{2} n + o(n)$. To highlight a key idea, however, we start with a slick proof that $\ept(G)\le n-1$ using Doob's Optional Stopping Theorem. Note that this result already improves on the bound $\ept(G)\le \frac{e}{e-1}\cdot n$ in Theorem \ref{snipedThm} (b).

\begin{theorem}\label{n-bound}
    We have $\ept(G,S)\le n-\abs{S}.$
\end{theorem}
\begin{proof}
    Let $X_0,X_1,\ldots$ denote random variables such that $X_i$ is the number of blue vertices at time $i$, where $X_0=\abs{S}$. Let the random variable $T$ denote the propagation time, namely the smallest index $i$ for which $X_i = n$. It is clear that $T$ is a valid stopping time because $T\le i$ for any fixed $i$ if and only if $X_i = n$. Consider the sequence of random variables $M_0,M_1,\ldots$ defined by $M_n = X_n - n,$ noting that $M_0 = X_0 - 0 = \abs{S}$. We claim that as long as not all the vertices are blue, then the expected number of blue vertices increases by at least 1 from the current step to the next step. This proves the claim.
    
    \begin{claim}
    Suppose that $X_i < n$. Then $\bE[X_{i+1}\mid X_i] \ge X_i + 1.$
    \end{claim}
    
    To prove the claim, we first note that there exists at least 1 blue vertex $u$ with a white neighbor. Letting $b\ge 0$ denote the number of blue neighbors of $u$ and $w\ge 1$ the number of white neighbors of $u,$ the probability that any given white neighbor of $u$ turns blue is at least $\frac{b+1}{b+w}\ge \frac{1}{w}.$ By linearity of expectation over all $w$ white neighbors, we see that the expected number of blue neighbors of the vertex $u$ increases by at least 1.
    
    Returning to the proof of Theorem \ref{n-bound}, we first claim that $\{M_i\}_{i=0}^\infty $ is a submartingale with respect to $\{X_i\}_{i=0}^\infty.$ Indeed, by the claim we have
    \begin{align*}
        \bE[M_{i+1}\mid X_0,X_1,\ldots,X_i]
        &= \bE[M_{i+1}\mid X_i]
        = \bE[X_{i+1}\mid X_i]-(i+1)
        \ge X_i - i
        = M_i.
    \end{align*}
    Note also that $\abs{M_i-M_{i-1}}$ is uniformly bounded across all $i\ge 0$ by $n+1$. Finally, the fact that $\bE[T]<\infty$ simply follows from $\ept(G,X_0)$ being finite, which can be deduced from the claim. Hence we can apply Doob's Optional Stopping Theorem to get
    \begin{align*}
        \bE[M_T]
        &= \bE[X_T] - \bE[T]
        = n - \bE[T]
        \ge \bE[M_0]
        = \abs{S}
    \end{align*}
    and $\ept(G,S) = \bE[T] \le n-\abs{S}$, as desired.
\end{proof}

\begin{corollary}\label{n-1-bound}
We have $\ept(G)\le n-1.$
\end{corollary}

\begin{remark}
Note that Theorem \ref{n-bound} is tight in the following sense. Fix a value of $\abs{S}.$ Then construct a graph $G$ that is a path of $n$ vertices from left to right. Let $S$ be the set consisting of the leftmost $\abs{S}$ vertices. Then $\ept(G, S) = n-\abs{S}$.
\end{remark}

We introduce a lemma that gives conditions on the current set of blue vertices regarding when the expected number of blue vertices increases by at least 2. Recall from the claim in Theorem \ref{n-bound} that as long as not all the vertices are blue, the expected number of blue vertices increases by at least 1 at the current step.

\begin{lemma}\label{2-lemma}
    %Let $S \subeq G$ be some subset of blue vertices. 
    Let $v_1, \dots, v_k$ be the blue vertices that are connected to at least 1 white vertex, and $w_1, \dots, w_\ell$ be the white vertices that are connected to at least 1 blue vertex. If $k, \ell \ge 3,$ then one of the following must be true:
\begin{itemize}
    \item With at most one exception, all of the $v_i$'s satisfy $\deg_w (v_i) = 1.$ In addition, some white vertex is connected to all of the $v_i$'s.
    \item The expected number of blue vertices after a single iteration of probabilistic zero forcing increases by at least 2.
\end{itemize}    
\end{lemma}

\begin{proof}
    Order the $v_i$'s in decreasing order of $\deg_w (v_i).$ First, we show that $\deg_w (v_1) \ge 3$ and $\deg_w (v_2) \ge 2$ implies that the expected number of new blue vertices is at least 2. Let $a := \deg_w (v_1)$, $b := \deg_w (v_2)$, and $c := \deg_w (v_3).$ Then, vertex $v_1$ will convert $a$ vertices $w_j$ to blue each with probability $\frac{1}{a},$ vertex $v_2$ will convert $b$ vertices $w_j$ to blue each with probability at least $\frac{1}{b} (1 - \frac{1}{a}),$ and vertex $v_3$ will convert $c$ vertices $w_j$ to blue each with probability $\frac{1}{c} \bigp{1 - \frac{1}{a}}\bigp{1 - \frac{1}{b}}.$ Here we are saying that vertex $v_i$ converts $w_j$ to blue if both $v_i$ propagates to $w_j$ and no $v_{i'}$ propagates to $w_j$ for any $i' < i.$ Therefore, since $a \ge 3$ and $b \ge 2,$ the expected number of additional blue vertices is at least
    \begin{align*}
        a \cdot \frac{1}{a} + b \cdot \frac{1}{b} \cdot \left(1 - \frac{1}{a}\right) + c \cdot \frac{1}{c} \cdot \left(1 - \frac{1}{a}\right)\left(1 - \frac{1}{b}\right)
        &= 1 + \left(1 - \frac{1}{a}\right) + \left(1 - \frac{1}{a}\right)\left(1 - \frac{1}{b}\right)
        \\&\ge 1 + \frac{2}{3} + \frac{2}{3} \cdot \frac{1}{2}
        \\&= 2.
    \end{align*}

    Assume the expected number of new blue vertices is less than 2. Then, either $\deg_w (v_1) \le 2$ or $\deg_w (v_2) = 1.$ In the latter case, we have that $\deg_w (v_i) = 1$ for all $i \ge 2.$ If any $v_i, v_{i'}$ for $i, i' \ge 2$ are connected to different white vertices, then both white vertices will be forced blue, contradicting our assumption. Thus, we can fix some $j$ and say $v_i$ is connected to $w_j$ for all $i \ge 2.$ If $v_1$ is not connected to $w_j,$ then $v_1$ will in expectation convert at least 1 white vertex to blue, and $w_j$ will be converted to blue with probability $1$. This contradicts our assumption, so $w_j$ must be connected to all $v_i$'s, even for $i = 1.$
    
    The final case is that $\deg_w (v_1) \le 2$, so $\deg_w (v_i) \le 2$ for all $i$. In this case, each $w_j$ will become blue with probability at least $\frac{1}{2},$ so we must have $\ell \le 3,$ and therefore $\ell = 3.$ Moreover, if some $v_i$ had $\deg_w (v_i) = 1,$ then some $w_j$ will become blue with probability $1$, which will make the expected number of new blue vertices at least 2. Therefore, $\deg_w (v_i) = 2$ for all $i$, so $\displaystyle \sum_{i = 1}^{k} \deg_w (v_i) = 2k.$ Since $\ell = 3,$ and $\deg_v (w_j) \le k$ for all $k$, there must be at least two indices $j$ such that $\deg_v (w_j) \ge 2,$ or else $\displaystyle \sum_{j = 1}^{\ell} \deg_v (w_j) \le k + 2 < 2k$ since $\ell = 3$ and $k \ge 3.$ However, for each index $j$ such that $\deg_v (w_j) \ge 2,$ we have the probability of $w_j$ becoming blue is at least $\frac{3}{4},$ and since $\ell = 3,$ we thus have the expected number of vertices that become blue is at least 2. This completes the proof.
\end{proof}

\begin{definition}
    Define a vertex $v$ of $G$ to be a \emph{cut-vertex} if removing $v$ and all edges with an endpoint at $v$ from $G$ causes $G$ to become disconnected.
\end{definition}

\begin{definition}
    Define a pair $(v, w)$ of vertices in $G$ to be a \emph{cornerstone} if the following two conditions hold:
\begin{enumerate}
    \item Removing $v, w$ and all edges with an endpoint at $v$ or $w$ from $G$ causes $G$ to become disconnected
    \item Either $v$ and $w$ are either connected by an edge or share a common neighbor in $G$.
\end{enumerate}
\end{definition}

\begin{definition}\label{cornerstone}
    If $v$ is a cut-vertex, consider all pairs of disjoint subsets $(S, T)$ such that $S \cup T = V \backslash \{v\}$ and there are no edges between $S$ and $T$. Define $g(v)$ to equal the minimum over all possible pairs of $\max(\abs{S}, \abs{T}).$ If $v$ is not a cut-vertex, define $g(v) = n-1.$
    
    Likewise, if $(v, w)$ is a cornerstone, consider all pairs of disjoint subsets $(S, T)$ such that $S \cup T = V \backslash \{v, w\}$ and there are no edges between $S$ and $T$. Define $g(v,w)$ to equal the minimum over all possible pairs of $\max(\abs{S}, \abs{T}).$ If $(v, w)$ are connected or have a common neighbor but do not form a cornerstone, define $g(v,w) = n-2.$
\end{definition}

We consider the following modified algorithm.

\begin{enumerate}
    \item Choose a single vertex $v$ or a pair of vertices $v, v'$ sharing either an edge or a common neighbor, such that the value of $g(v)$ or $g(v, v')$ in Definition \ref{cornerstone} is minimized. If we chose a single vertex, pick sets $S, T$ such that $\abs{S} \le \abs{T} \le g(v),$ $S \cap T = \emptyset,$ $S \cup T = G \backslash \{v\}$, and there are no edges between $S$ and $T$. Likewise, if we chose a pair of vertices, pick $S, T$ in the same way except that $S \cup T = G \backslash \{v, v'\}.$ Note that $S$ may be empty, if there are no cut-vertices or cornerstones.
    \item Initialize with $v$ blue and all other vertices white.
    \item Pick some arbitrary ordering $v_1, \dots, v_n$ of the vertices, where $v, v'$ can be labeled with any number.
    \item Run the probabilistic zero forcing process until all neighbors of $v$ are blue, and, if we picked a cornerstone or pair of vertices $v, v'$ in Step (1), until all neighbors of $v'$ as well.
    %\item Remove vertices $v$ and $w$.
    \item Turn all vertices that are not $v, v',$ or any of their neighbors white.
    \item Run the probabilistic zero forcing algorithm on the induced subgraph $G[T]$, until the number of white vertices in $T$ is at most $\abs{S}+3$.
    \item At each step, suppose there is some blue vertex in $G[T]$ with $k \ge 1$ white neighbors in $G[T]$. Then, choose such a blue vertex $v_i \in G[T]$ with the smallest index $i$, and run probabilistic zero forcing but where each white neighbor of $v_i$ becomes blue with probability $1$ if $v_i$ only has 1 white neighbor, and becomes blue with probability $\frac{4}{3k}$ otherwise. If there is no such vertex in $G[T]$, do nothing.
    Run this same procedure on $G[S]$ simultaneously.
    %Do the same thing for $G[S]$ during the same step.
\end{enumerate}

Note that by Lemma \ref{GenesonCoupling}, Step (5) will only increase the total expected propagation time. Hence it suffices to bound the expected runtimes of Steps (4), (6), and (7) and show that their sum is at most $\frac{n}{2} + O(\log n),$ which will imply the desired bound on expected propagation time.

\begin{lemma}\label{step-4}
    Step \emph{(4)} takes $O(\log n)$ time in expectation.
\end{lemma}

\begin{proof}
    By Lemma \ref{RadiusBash}, any vertex $w$ of distance at most 3 from $v$ will be blue after $C \log n + t$ steps with probability at least $1 - \beta^t$. Therefore, for some $C' > C,$ after $C' \log n + t$ steps, each vertex $w$ of distance at most 3 from $v$ will be blue with probability at least $1 - \frac{1}{n} \cdot \beta^t,$ so the probability that all such vertices are blue is at least $1 - \beta^t.$ Since $v'$ has distance at most 2 from $v$, all neighbors of both $v$ and $v'$ will be blue if all vertices of distance at most 3 from $v$ are blue. Therefore, the expected time is $O(\log n)$.
\end{proof}

\begin{lemma}\label{step-6}
    Step \emph{(6)} takes at most $\frac{1}{2}(\abs{T}-\abs{S})$ steps in expectation.
\end{lemma}

\begin{proof}
    Suppose we are running the algorithm on $G[T]$ and the number of white vertices before some iteration is $k \ge \abs{S}+3.$ We show that the expected number of new blue vertices after the next step will be at least 2.
    
    If there is exactly 1 blue vertex $w$ connected to any white vertex in $G[T]$, then if we remove $w,$ the white vertices in $G[T]$ and blue vertices will be disconnected. Moreover, the white vertices in $G[T]$ do not share edges with $v$ or $v',$ or else they would have been colored blue at Step (3). Finally, $S$ and $T$ are disconnected, so in fact removing $w$ from the original graph $G$ causes the white vertices in $T$ to be disconnected from all other vertices in $G.$ However, the number of white vertices in $G[T]$ is at least $\abs{S}+3$ and the other disconnected component contains $S$ and $v$, so $w$ is a cut-vertex with $g(w) > \abs{S},$ which is a contradiction because $v$ was chosen among all cut-vertices and cornerstones to minimize the value of the function $g.$ Likewise, if there is exactly 1 white vertex $w$ connected to any blue vertices in $G[T],$ $w$ will form a cut-vertex for the same reason, as the remaining white vertices in $G[T]$ cannot be connected to any other vertices in $G$. However, the number of white vertices is at least $\abs{S}+3$ and the other disconnected component contains $S$ and $v$, so $w$ is a cut-vertex with $g(w) > \abs{S},$ which is a contradiction.
    
    If we have exactly two blue vertices $w_1, w_2$ connected to any white vertices in $G[T]$, then removing $w_1$ and $w_2$ will cause the white vertices in $G[T]$ to be disconnected from the rest of the graph, by the same argument as in the previous paragraph. This means that if $w_1$ and $w_2$ have a common neighbor, then $g(w_1, w_2) > \abs{S},$ which is a contradiction. Otherwise, $w_1$ and $w_2$ have no common neighbor, so each of $w_1$ and $w_2$ turns at least 1 white vertex blue in expectation.
    
    Otherwise, there are at least 3 blue vertices connected to any white vertices in $G[T].$ Suppose there are at most 2 white vertices connected to any blue vertices in $G[T].$ The first case is that each blue vertex is only connected to exactly 1 white vertex, in which case these 2 white vertices will turn blue and we have an expected increase of 2 in the number of new blue vertices. The second case is that there exists a blue vertex connected to 2 white vertices $x_1, x_2$, in which case the two white vertices form a cornerstone. Since there are at least $\abs{S}+3$ white vertices remaining assuming that we are not done with Step (6), we have $g(x_1, x_2) > \abs{S},$ contradiction.
    
    The last case is that there are at least 3 blue vertices connected to white vertices and at least 3 white vertices connected to blue vertices. By Lemma \ref{2-lemma}, we are done unless all blue vertices except for a blue vertex $w_1$ is connected to exactly 1 white neighbor, and there is some white vertex $x_1$ connected to all blue vertices. What this means is that $w_1$ is connected to all white vertices $x_1,x_2,\ldots,x_k$ that are adjacent to blue vertices and that none of $x_2,x_3,\ldots,x_k$ are connected to a blue vertex that is not $w_1.$ However, this implies that $(w_1, x_1)$ is a cornerstone with $g(w_1,x_1) > \abs{S},$ which is a contradiction.
    
    We conclude that in all cases the expected number of blue vertices after the next step will be at least 2. The result that Step (6) takes at most $\frac{1}{2} (\abs{T} - \abs{S})$ follows from a Doob's Optional Stopping Theorem argument that is the same as the proof of Theorem \ref{n-bound}. In particular, we use a submartingale that at time $t$ is defined to be $X_t - 2t,$ where we recall that $X_t$ is the number of blue vertices at time $t.$
\end{proof}

\begin{lemma}\label{step-7}
    Step \emph{(7)} takes $\abs{S} + O(1)$ steps in expectation.
\end{lemma}

\begin{proof}
    First, consider the process only on the $G[T]$ side. Suppose that there are $m$ white vertices in $G[T]$ at the beginning of step (7). Let $C \approx 1.8328$ be a solution to the equation $e^{4/3 \cdot (1 - 1/C)} = C$ over $\bR$. If we let $T_t$ be the set of blue vertices in $T$ after $t$ steps and let $X_t = \abs{T_t},$ we consider the supermartingale $Y_t := C^{t-X_t}$ with respect to $T_0, T_1, \dots$.
    
    To show that this process is a supermartingale, first suppose that $T_0, \dots, T_t$ are known and $T_t$ is such that the blue vertex selected at the $(t+1)^{\text{th}}$ step has exactly 1 white neighbor. Then, $X_{t+1} = X_t+1$ so in fact $C^{(t+1)-X_{t+1}} = C^{t-X_t}.$ Otherwise, if the blue vertex selected has exactly $k \ge 2$ white neighbors, then since each white neighbor becomes blue with probability $\frac{4}{3k}$ independently, we have
\begin{align*}
\bE[C^{(t+1)-X_{t+1}} \mid T_0, T_1, \dots, T_t] &= C^{(t+1)-X_t} \cdot \left(\left(1 - \frac{4}{3k}\right) + \frac{4}{3k} \cdot \frac{1}{C}\right)^{k} \\
&= C^{(t+1)-X_t} \cdot \left(1 - \frac{4 (1 - \frac{1}{C})}{3k}\right)^k \\
&\le C^{(t+1)-X_t} \cdot e^{-4/3 \cdot (1 - 1/C)} \\
&= C^{(t+1)-X_t} \cdot \frac{1}{C} = C^{t-X_t}.
\end{align*}
    
    Unfortunately, the function $Y_t = C^{t-X_t}$ can have $\abs{Y_{t+1}-Y_t}$ arbitrarily large, so we cannot directly use Doob's Optional Stopping Theorem. However, we note that if we force the algorithm to stop after $\ell$ steps, then we will be able to use Doob's Optional Stopping Theorem. Namely, if $\tau$ is the amount of time needed until all vertices in $T$ are blue, then Doob's Optional Stopping Theorem will give us that for any $\ell \in \bN,$
\[\bE[Y_{\min(\tau, \ell)}] = \bE\left[C^{\min(\tau, \ell) - X_{\min(\tau, \ell)}}\right] \le Y_0 = C^{-X_0},\]
    where $X_0$ is the number of blue vertices in $T$ at the beginning of the process. Since $X_{\min(\tau, \ell)} \le \abs{T},$ we thus have
$\bE[C^{\min(\tau, \ell) - \abs{T}}] \le C^{-X_0}$, which implies $\bE[C^{\min(\tau, \ell)}] \le C^{\abs{T}-X_0}$.
    Therefore, by setting $\ell = \abs{T}-X_0 + k$ for some $k \in \bN,$ we obtain by Markov's Inequality that $\bP(\tau \ge \abs{T}-X_0+k) \le C^{-k},$ so we have $\bE[\max(0, \tau-\abs{T}+X_0)] \le D$ for some fixed constant $D$. Since $\abs{T}-X_0 \le \abs{S}+3,$ and the number of white vertices in $G[T]$ at the beginning of step (7) was at most $\abs{S}+3,$ we must have $\bE[\max(0, \tau-(\abs{S}+3))] \le D,$ so $\bE[\max(0, \tau-\abs{S})] \le D+3.$
    
    If we run the same procedure on $G[S]$ and call this stopping time $\tau',$ we likewise obtain $\bE[\max(0, \tau'-\abs{S})] \le D$ for the same value of $D.$ Therefore,
\begin{align*}
\bE[\max(\tau, \tau)] &\le \bE[\max(\tau, \tau', \abs{S})] \\
&= \abs{S}+\bE[\max(0, \tau-\abs{S}, \tau'-\abs{S})] \\ 
&\le \abs{S}+\bE[\max(0, \tau-\abs{S})] + \bE[\max(0, \tau'-\abs{S})] \\ 
&\le \abs{S}+2D+3.
\end{align*}
    This completes the proof, since the number of steps equals $\max(\tau, \tau').$
\end{proof}

We observe that by Lemmas \ref{GenesonCoupling} and \ref{OurCoupling}, Step 5 and only doing the propagation from certain vertices with equal or lower probabilities in Steps (6) and (7) cannot decrease the expected propagation time. Linearity of expectation, along with Lemma \ref{step-4}, Lemma \ref{step-6}, Lemma \ref{step-7}, and our observation, yields the following theorem.

\begin{theorem}\label{n/2-bound}
    We have $\ept(G) \le \frac{n}{2} + O(\log n)$.
\end{theorem}

\begin{proof}
    It suffices to observe that
    \begin{align*}
        O(\log n) + \frac{1}{2}\bigp{\abs{T} - \abs{S}} + \abs{S} + O(1)
        &= O(\log n) + \frac{1}{2} \bigp{\abs{T} + \abs{S}}
        = \frac{n}{2} + O(\log n).
    \end{align*}
\end{proof}

\begin{remark}
The factor of $\frac{1}{2}$ in Theorem \ref{n/2-bound} is tight, as demonstrated by the expected propagation time of a path $P_n$ of $n$ vertices, which has been computed to be
$$
    \ept(P_n) = \begin{cases} \frac{n}{2} + \frac{2}{3} & n\equiv 0\bmod 2 \\ \frac{n}{2} + \frac{1}{2} & n\equiv 1\bmod 2.\end{cases}
$$
\end{remark}

\subsection{A lower bound on expected propagation time}\label{lower}

\begin{theorem} \label{lowerbound}
    For a subset $S \subeq V$ of size $k$, we have $\ept(G, S) \ge \log_2 \log_2 (2n) - \log_2 \log_2 (2k).$
\end{theorem}

\begin{proof}
    Note that if at some point in time, there are $k$ blue vertices, then each blue vertex $v$ has at most $k-1$ blue neighbors, and will in expectation, convert at most $k$ points blue. Therefore, the expected number of blue vertices after a single round of probabilistic zero forcing is at most $k+k^2 \le 2k^2,$ regardless of which $k$ vertices were blue.
    
    Let $B_t$ represent the number of blue vertices after $t$ steps, so for example $B_0 = \abs{S}.$ We note that $A_t := \log_2 \log_2 (2B_t) - t$ is a supermartingale. To see why, note that 
\begin{align*}
\bE[A_{t+1} \mid A_0, \dots, A_t] &= \bE[\log_2 \log_2 (2B_{t+1}) \mid B_0, \dots, B_t] - (t+1) \\
&\le \log_2 \log_2 \left(2 \bE[B_{t+1} \mid B_0, \dots, B_t]\right) - (t+1) \\
&\le \log_2 \log_2 (4B_t^2) - (t+1) \\
&= \log_2 \log_2 (2B_t) - t
\\&= A_t.
\end{align*}
    Here, we are using Jensen's Inequality and that $\log_2 \log_2 (2x)$ is concave on the interval $[1, \infty)$.
    
    We know that the $\ept(G, S)$ is finite (in fact, at most $n - \abs{S}$) and that $\abs{A_{t+1}-A_t}$ is absolutely bounded by $$1 + \abs{\log_2 \log_2 (2B_{t+1}) - \log_2 \log_2 (2B_{t})} \le 1 + \log_2 \log_2 (2n).$$ Therefore, we can apply Doob's Optional Stopping Theorem to get that if $\tau$ represents the total number of steps, then $$\bE[A_{\tau}] \le \bE[A_0] = \log_2 \log_2 (2\abs{S}).$$ However, since $A_{\tau} = \log_2 \log_2 (2n) - \tau$, we have $$\log_2 \log_2 (2n) - \bE[\tau] \le \log_2 \log_2 (2\abs{S}),$$ so $$\bE[\tau] \ge \log_2 \log_2 (2n) - \log_2 \log_2 (2\abs{S}).$$
    
    Since $\log_2 \log_2 (2 \cdot 1) = 0,$ we have $\ept(G) \ge \log_2 \log_2 n.$ Also, since $\log_2 \log_2 (2k) \le k$ for all $k \ge 1$, we have that $\abs{S}+\ept(G, S) \ge \log_2 \log_2 (2n),$ implying that $\thpzf(G) \ge \log_2 \log_2 (2n).$
\end{proof}

\begin{corollary}
We have $\thpzf(G) \ge \log_2 \log_2 (2n)$ and $\ept(G) \ge \log_2 \log_2 (2n)$.
\end{corollary}

\begin{proof}
    This follows directly from $\ept(G)\le \thpzf(G).$
\end{proof}

\section{Further directions} \label{Conclusion}
Here we list future directions and open problems that arise in probabilistic zero forcing.

\begin{enumerate}
    \item If one believes a path $P_n$ to be the graph on $n$ vertices with the maximum expected propagation time, one might guess that in general the expected propagation time is at most a constant added to $\frac{n}{2}.$ We leave the interested reader with the following conjecture.

    \begin{conjecture}\label{n/2-conj}
        We have $\ept(G) \le \frac{n}{2} + O(1).$
    \end{conjecture}
    
    \item It is very likely that one can get better bounds on the probabilistic zero forcing throttling number $\thpzf(G).$ There is the following theorem by Geneson and Hogben.
    \begin{theorem}[{\cite[Theorem 6.5]{GH}}]
        Across all graphs $G$, the maximum possible probabilistic throttling number is $\Omega(\sqrt{n})$ and $O(\sqrt{n} \cdot \log^2 n).$
    \end{theorem}
    We conjecture that the throttling number is actually $\Theta(\sqrt{n})$ across all graphs $G$. An approachable problem that makes progress toward this conjecture could be the following.
    \begin{conjecture}\label{t}
        Prove that the maximum possible probabilistic throttling number is $\Theta(\sqrt{n})$ among connected \emph{trees}.
    \end{conjecture}
    We remark that a path graph achieves the lower bound, as proven in \cite[Proposition 6.3]{GH}. It seems very likely that a method that involves rooting a tree and removing sub-trees of size slightly bigger than $\sqrt{n}$ can solve Conjecture \ref{t}.

\iffalse    
    \item Let $G(n,p)$ denote a random graph such that each edge between two vertices is present independently with probability $p.$ Geneson and Hogben show that with high probability $\ept(G(n,p)) = O(\log^2 n)$ and with high probability $\thpzf(G(n,p)) = O(\log n \cdot \log \log n).$ We make the following strong conjecture.
    \begin{conjecture}
    With high probability we have $\ept(G(n,p)) = (1+o(1))\cdot \log \log n.$
    \end{conjecture}
    Note that this would also imply $\thpzf(G(n,p)) = (1+o(1))\cdot \log \log n.$ An intermediate step would be to prove that $\ept(G(n,p)) = O(\log \log n).$ A related result is the following theorem by Chan et al.\! in \cite{sniped}.
    \begin{theorem}[\cite{sniped}, Theorem 3.1]
        For any positive integer $n$ we have $\ept(K_n) = \Theta(\log \log n),$ where $K_n$ is the complete graph on $n$ vertices.
    \end{theorem}
    The lower bound of our conjecture is immediate from Theorem \ref{lowerbound}, and it is conceivable that similar upper bound arguments in the proof of Theorem 3.1 of \cite{sniped} could establish $\ept(G(n,p)) = O(\log \log n)$ as well.
\fi

    \item Geneson and Hogben \cite{GH}, as well as Chan et al.\! \cite{sniped}, compute exactly or establish bounds for the expected propagation time of specific graphs $G,$ such as paths \cite{GH}, cycles \cite{GH}, spider graphs \cite{GH}, star graphs \cite{GH}, complete bipartite graphs \cite{sniped}, ``sun'' graphs \cite{sniped}, and ``comb'' graphs. In addition, English et al. \cite{snipedround2} provided strong bounds for Erd\H{o}s--Renyi random graphs $G(n, p)$ for a wide range of $p$.
    
    An interesting problem would be to compute or estimate the expected propagation time for the following classes of graphs:
    \begin{itemize}
        %\item Erd\H{o}s--R\'{e}nyi graphs $G(n,p)$ with $p$ either constant independent of or a decreasing function of $n.$
        \item $d$-regular graphs.
        \item Product graphs. In particular, given two graphs $G$ and $H,$ can one say anything about $\ept(G\times H)$ in terms of $\ept(G)$ and $\ept(H)$?
    \end{itemize}
    
\end{enumerate}

\end{document}